\documentclass[11pt,reqno,oneside]{amsart}
\usepackage{amsmath,amsthm,amssymb, amsfonts, mathrsfs}\usepackage{plain}
\usepackage{hyperref}
\numberwithin{equation}{subsection}

\numberwithin{equation}{section} 
\numberwithin{equation}{section} 
\theoremstyle{plain}
\newtheorem{theorem}{Theorem}[section]

\newtheorem{corollary}{Corollary}[section]
\newtheorem{lemma}{Lemma}[section]

\newtheorem{remark}{Remark}[section]

\theoremstyle{definition}
\newtheorem{definition}{Definition}[section]
\numberwithin{equation}{section}

\newcommand{\U}{\mathbb U}

\makeatletter
\let\c@equation\c@figure
\makeatother
 \makeatletter

\begin{document}
\flushbottom
\title[]{Inclusion and Majorization Properties of Certain Subclasses
of Multivalent Analytic Functions Involving a Linear Operator.}
\author{A.K.Sahoo }\address{$^{1}$Department of Mathematics,\\
VSS University of Technology,\\Sidhi Vihar,
Burla, Sambalpur-768017,India}\email{ashokuumt@gmail.com}

\begin{abstract}
The object of the present paper is to study certain properties and characteristics of
the operator $Q_{p,\beta}^{\alpha}$defined on p-valent analytic function by using technique of differential subordination.We also obtained result involving majorization problems by applying the operator to p-valent analytic function.Relevant connection of the the result are presented here with those obtained by earlier worker are pointed out.
\end{abstract}

 \maketitle
\section {Introduction and preliminaries}
Let $\mathcal A_p(n)$ denote the class of functions of the form
\begin{equation}\label{eq2:eq1}
 f(z)=z^p+\sum_{k=n}^{\infty}a_{p+k}z^{p+k}\qquad (p,n \in \mathbb N=\{1,2,\dots\})
\end{equation}
 which are analytic and $p$-valent in the open unit disk $\mathbb U=\{z
 \in \mathbb C : |z|< 1\}.$ For convenience, we write $\mathcal
 A_p(1)=\mathcal A_p,\; \mathcal A_1(n)=\mathcal A(n)$ and  $\mathcal
 A_1(1)=\mathcal A.$

  For the function $f$, given by \eqref{eq2:eq1} and the function $g$ defined in
$\mathbb U$ by \[ g(z)=z^p+\sum_{k=n}^{\infty}b_{p+k}z^{p+k},\;\]
 the Hadamard product (or convolution) of $f$ and $g$ is given by
\begin{equation}\label{n2.1.2}
(f \star g)(z)=z^{p}+\sum_{k=n}^{\infty}a_{p+k}b_{p+k}z^{p+k}=(g
\star f)(z) \quad (z \in \mathbb U).
 \end{equation}

   For $p \in \mathbb N,\; m \in \mathbb N_{0}=\mathbb N \cup \{0\},\;
\alpha \in \mathbb R,\; \beta > 0$ with $\alpha+p\beta > 0$,\; Swamy
\cite{SRS1}( see also \cite{SRS2}) introduced and studied a linear
operator $I_{p,\alpha,\beta}^m: \mathcal A_p \longrightarrow
\mathcal A_p$ defined as follows:
\begin{align*}
I_{p,\alpha,\beta}^{0}f(z)=f(z), \quad
I_{p,\alpha,\beta}^{1}f(z)=I_{p,\alpha,\beta}f(z)=\dfrac{\alpha
f(z)+\beta zf^{\prime}(z)}{\alpha+p\beta},
\end{align*}
and, in general
\begin{align*}
I_{p,\alpha,\beta}^{m}f(z)=I_{p,\alpha,\beta}\left(I_{p,\alpha,\beta}^{m-1}f(z)\right)\quad
(m \in \mathbb N_{0}; z \in \mathbb U).
\end{align*}

If $f(z)=z^p+\sum_{k=p+1}^{\infty}a_{k}z^{k} \in \mathcal A_p$, then
it follows from the definition of the operator
$I_{p,\alpha,\beta}^{m}$ that
\begin{equation}\label{eq2:eq3}
I_{p,\alpha,\beta}^{m}f(z)=z^p+\sum_{k=p+1}^{\infty}\left(\dfrac{\alpha+k
\beta}{\alpha+p \beta}\right)^{m}a_{k}z^{k}\quad (m \in \mathbb
N_0=\mathbb N\cup\{0\}; z \in \mathbb U).
\end{equation}
Since $\alpha+k\beta > \alpha+p\beta >0$ for $k \geq p+1$, the
operator can be defined for negative integral values of  $m$ as
\begin{equation}\label{eq2:eq4}
I_{p,\alpha,\beta}^{-m}f(z)=z^p+\sum_{k=p+1}^{\infty}\left(\dfrac{\alpha+p
\beta}{\alpha+k\beta}\right)^{m}a_{k}z^{k}\quad (m \in \mathbb N_0;
z \in \mathbb U).
\end{equation}
 We, further observe that
\begin{align*}
I_{p,\alpha,\beta}^{-1}f(z)=\dfrac{\frac{\alpha}{\beta}+p}{z^{\frac{\alpha}{\beta}}}\int_0^z
t^{\frac{\alpha}{\beta}-1}f(t)dt=I_{p,\alpha,\beta}^{-1}\left(\dfrac{z^p}{1-z}\right)\star
f(z)\quad (z \in \mathbb U)
\end{align*}
so that
\begin{equation*}
I_{p,\alpha,\beta}^{-m}f(z)=\underbrace{I_{p,\alpha,\beta}^{-1}\left(\dfrac{z^p}{1-z}\right)
\star I_{p,\alpha,\beta}^{-1}\left(\dfrac{z^p}{1-z}\right)\star
\dots
I_{p,\alpha,\beta}^{-1}\left(\dfrac{z^p}{1-z}\right)}_{m-times}\star
f(z)\;(z \in \mathbb U).
\end{equation*}
Thus, in view of \eqref{eq2:eq3} and \eqref{eq2:eq4}, we define a
linear operator
\begin{equation*}
\Theta_{p}^{m}(n;\alpha, \beta): \mathcal A_{p}(n) \longrightarrow
\mathcal A_{p}(n)
\end{equation*} by
\begin{equation}\label{eq2:eq5}
\Theta_{p}^{m}(n;\alpha,\beta)f(z)=z^p+\sum_{k=p+n}^{\infty}
\left(\dfrac{\alpha+k \beta}{\alpha+p \beta}\right)^{m}a_{k}
z^{k}\quad (z \in \mathbb U)
\end{equation}
where $p,\; n \in \mathbb N,\; m \in \mathbb Z=\{\dots, -2, -1, 0,
1, 2, \dots\}, \alpha \in \mathbb R$ and $\beta > 0$ with
$\alpha+p\beta
>0.$ From \eqref{eq2:eq5}, it is easily seen that
\begin{equation}\label{eq2:eq6}
\beta\,z\left(\Theta_{p}^{m}(n;\alpha,\beta)f\right)^{\prime}(z)
=(\alpha+p\beta)\Theta_{p}^{m+1}(n;\alpha,\beta)f(z)-\alpha\,\Theta_{p}^{m}(n;\alpha,\beta)f(z)\quad
(m \in \mathbb Z; z \in \mathbb U).
\end{equation}
For convenience we write
\begin{equation}\label{2.1.6}
\Theta_{p}^{m}\left(1;\alpha,\beta\right)f(z)=\Theta_{p}^{m}\left(\alpha,\beta\right)f(z)
\;(z\in\mathbb U).\end{equation}
  Further, by suitably specializing the parameters $p,\;\alpha$ and $\beta$ in \eqref{2.1.6},  we
obtain the following linear operators studied earlier by various
authors.
\begin{enumerate}
\item[(i)]\;$\Theta_{1}^{m}(\alpha,\beta)f(z)=\mathcal I_{\alpha,
\beta}^{m}f(z)\quad (f \in \mathcal A; m \in \mathbb N_0)$\;(see
Swamy \cite{SRS2});
\item[(ii)]\;$\Theta_{p}^{m}\left(\alpha,\beta\right)f(z)=I_{p,\alpha,\beta}^{m}f(z)\quad (f \in \mathcal A_p
\; m \in \mathbb N_0)$ \;(see Swamy \cite{SRS1});
\item[(iii)]\;$\Theta_{p}^{m}(\ell+p-p\lambda,\lambda)f(z)=\mathcal
J_p^{m}(\lambda, \ell)\quad (f \in \mathcal A_p, \ell \geq 0,
\lambda>0; m \in \mathbb Z)$\;(see
C$\check{\text{a}}$tas \cite{ACatas});
\item[(iv)]\;$\Theta_{p}^{m}(\alpha,1)f(z)=\mathcal
I_{p}^{m}(\alpha)f(z)\quad (f \in \mathcal A_p, \alpha > -p; m \in
\mathbb N_0)$\;(see Aghalary \cite{Aghalary}, Shivaprasad et al.
\cite{SHV}, Srivastava et al. \cite{SSA});
\item[(v)]\;$\Theta_{p}^{m}(0,\beta)f(z)=\mathcal D_p^{m}f(z)\quad (f \in
\mathcal A_p, \beta > 0; m \in \mathbb N_0)$\;(see Aouf et al.
\cite{AMo}, Kamali et al. \cite{KO}, Orhan et al.\; \cite{HH});
\item[(vi)]\;$\Theta_{1}^{m}(\ell, 1)f(z)=\mathcal I_{\ell}^{m}f(z)\quad
(\ell > 0, m \in \mathbb N_{0})$\;(see Cho and Kim \cite{chokim},
Cho and Srivastava \cite{chosri});
\item[(vii)]\;$\Theta_{1}^{m}(1-\lambda, \lambda)f(z)=\mathcal
D_{\lambda}^{m}f(z)\quad (\lambda \geq 0, m \in \mathbb
N_{0})$\;(see Al-Oboudi \cite{fmal}), which yields the operator
$\mathcal D^{m}$ studied by Salagean \cite{sal}, for $\lambda=0$.
\end{enumerate}
Using the operator $\Theta_{p}^{m}(n;\alpha,\beta)$, we now define
\begin{definition}\label{def1}
For fixed parameters $A,B\,(-1 \leq B < A \leq 1),\; \beta > 0,
\;\mu\geq 0$ and $\alpha+p\beta > 0$, we say that a function $f \in
\mathcal A_p(n)$ is in the class $\mathcal S_{p,n}^{m}(\alpha,
\beta, \mu, A, B)$, if it satisfies the following subordination
condition:
\begin{equation}\label{eq2:eq7}
(1-\mu)\dfrac{\Theta_{p}^{m+1}(n;\alpha,\beta)f(z)}{\Theta_{p}^{m}(n;\alpha,\beta)f(z)}
+\mu\dfrac{\Theta_{p}^{m+2}(n;\alpha,\beta)f(z)}{\Theta_{p}^{m+1}(n;\alpha,\beta)f(z)}
\prec \dfrac{1+Az}{1+Bz}\quad (m \in \mathbb Z;z \in \mathbb U).
\end{equation}
\end{definition}

For ease of notation, we write
\begin{enumerate}
\item[(i)]\;$\mathcal S_{p,1}^{m}(\ell+p-p\lambda,\lambda, \mu; A, B)=\mathcal S_{p}^{m}(\lambda,\ell,\mu, A, B)$, the class of
functions $f \in \mathcal A_p$ satisfying the subordination
condition:
$$ (1-\mu)\dfrac{\mathcal J_p^{m+1}(\lambda, \ell)f(z)}{\mathcal
J_p^{m}(\lambda, \ell)f(z)} +\mu\dfrac{\mathcal J_p^{m+2}(\lambda,
\ell)f(z)}{\mathcal J_p^{m+1}(\lambda, \ell)f(z)} \prec
\dfrac{1+Az}{1+Bz}\quad (\lambda >0, \ell > -p, m \in \mathbb Z;z
\in \mathbb U);$$
\item[(ii)]\;$\mathcal S_{p,1}^{m}(\alpha, \beta, 1, A, B)=\mathcal
S_{p}^{m}(\alpha, \beta, A, B)$, the class of functions $f \in
\mathcal A_p$ satisfying the subordination
condition:$$\dfrac{\Theta_{p}^{m+2}(\alpha,\beta)f(z)}{\Theta_{p}^{m+1}(\alpha,\beta)f(z)}
\prec \dfrac{1+Az}{1+Bz}\quad ( z \in \mathbb U);$$
\item[(iii)]\;$\mathcal S_{p,1}^{m}(\alpha, \beta, \mu, A, B)=\mathcal
S_{p}^{m}(\alpha, \beta, \mu, A, B)$, the class of functions $f \in
\mathcal A_p$ satisfying the subordination condition
\eqref{eq2:eq7};
\item[(iv)]\;$\mathcal S_{p,1}^{m}\left(\alpha, \beta,\mu,
1-2\rho,-1\right)=\mathcal S_{p}^{m}\left(\alpha, \beta,\mu,
\rho\right)$, the class of functions $f \in \mathcal A_p$ satisfying
$$\text{Re}\left\{(1-\mu)\dfrac{\Theta_{p}^{m+1}(\alpha,\beta)f(z)}{\Theta_{p}^{m}(\alpha,\beta)f(z)}
+\mu\dfrac{\Theta_{p}^{m+2}(\alpha,\beta)f(z)}{\Theta_{p}^{m+1}(\alpha,\beta)f(z)}\right\}>
\rho\quad (0 \leq \rho < 1; z \in \mathbb U);$$
\item[(v)]\;$\mathcal S_{p,1}^{m}\left(\alpha, \beta,0,
1-2\rho,-1\right)=\mathcal S_{p}^{m-1}\left(\alpha, \beta, 1,
1-2\rho,-1\right)=\mathcal S_p^{m}(\alpha, \beta;\rho)$, the class
of functions $f \in \mathcal A_p$ satisfying
$$\text{Re}\left\{\dfrac{\Theta_p^{m+1}(\alpha,
\beta)f(z)}{\Theta_p^{m}(\alpha, \beta)f(z)}\right\}
> \rho\quad (0 \leq \rho < 1; z \in \mathbb U);$$
\item[(vi)]$\mathcal S_{p,1}^{m-1}\left(\alpha, \beta,1,\dfrac{A\beta(p-\eta)+B(\alpha+\beta\eta)}{\alpha+p\beta},B\right)$\\
$=\mathcal S_{p,1}^{m}\left(\alpha,
\beta,0,\dfrac{A\beta(p-\eta)+B(\alpha+\beta\eta)}{\alpha+p\beta},B\right)
 =\mathcal S_{p,\alpha,\beta}^{m}(\eta;A,B)$\\
The class $\mathcal S_{p,\alpha,\beta}^{m}(\eta;A,B)$ of functions
$f \;\text{in}\;\; \mathcal A_p$ satisfying the subordination
condition:$$\dfrac{1}{p-\eta}\left\lbrace
\dfrac{z\left(\Theta_p^{m}(\alpha,
\beta)f\right)^{\prime}(z)}{\Theta_p^{m}(\alpha,
\beta)f(z)}-\eta\right\rbrace \prec \dfrac{1+Az}{1+Bz}\quad (z \in
\mathbb U).$$ This class was introduced  and studied by Swamy
\cite{SRS1}, which in turn yields the class $\mathcal
S_{\alpha,\beta}^{m}(\eta;A,B)$ studied in \cite{SRS2} for $p=1$.
\item[(vii)]\;$\mathcal S_{p,1}^{0}(0, \beta, 0, A, B)=\mathcal
S_{p}^{\ast}(A, B)$, the class of functions $f \in \mathcal A_p$
satisfying the subordination
condition:$$\dfrac{zf^{\prime}(z)}{f(z)} \prec
\dfrac{p(1+Az)}{1+Bz}\quad (z \in \mathbb U).$$
\item[(viii)]\;$\mathcal S_{p,1}^{0}(0, \beta, 1, A, B)=\mathcal
S_{p}^{1}(0, \beta, 0; A, B)=\mathcal C_p(A,B)$, the class of
functions $f \in \mathcal A_p$ satisfying the subordination
condition:
$$1+\dfrac{zf^{\prime\prime}(z)}{f^{\prime}(z)}
\prec \dfrac{p(1+Az)}{1+Bz}\quad (z \in \mathbb U).$$
\end{enumerate}

  In the present investigation, we introduce a subclass $\mathcal S_{p,n}^{m}(\alpha,
\beta, \mu, A, B)$ of $\mathcal A_{p}(n)$. We derive certain
inclusion relationships, some useful characteristics and
majorization properties for the class $\mathcal S_{p,n}^{m}(\alpha,
\beta, \mu, A, B)$. The results obtained here in addition to
generalizing some of the work of Patel et al.~\cite{patel2} and
MacGregor~\cite{mag1} improves the corresponding work of
Swamy~\cite{SRS1}. We also obtain a number of new results for
functions belonging to this class in terms of subordination and the
various subclasses obtained as special cases of the class $\mathcal
S_{p,n}^{m}(\alpha, \beta, \mu, A, B).$
\section{Preliminary Lemmas}

To derive our main results, we shall need the following lemmas.
\begin{lemma} \cite{DJR}, see also \cite[p.71]{MM2}.\label{lem1}
Let $h$ be an analytic and convex (univalent) function in $\mathbb
U$ with $h(0)=1$  and $\phi$ be given by
\begin{equation} \label{eq2.2:eq1}
\phi(z)=1+c_{n}z^{n}+c_{n+1}z^{n+1}+\dots \quad (n \in \mathbb N; z
\in \mathbb U).
\end{equation}
If
\begin{equation} \label{eq2.2:eq2}
\phi(z)+\dfrac{z \phi^{\prime}(z)}{\gamma} \prec h(z)\quad
(\text{Re}(\gamma) \geq 0, \gamma \ne 0; z \in \mathbb U),
\end{equation}
then
\begin{equation*}
\phi(z) \prec \psi(z)=\dfrac{\gamma}{n}z^{-\frac{\gamma}{n}}\int_0^z
t^{\frac{\gamma}{n}-1}h(t)dt \prec h(z)\quad (z \in \mathbb U)
\end{equation*}
and the function $\psi$ is the best dominant of \eqref{eq2.2:eq2}.
\end{lemma}

We recall the definition of the class $\mathcal P(\gamma)\;(0 \leq
\gamma < 1)$(cf., Section 1.2) consisting of all functions of the
form
\begin{equation}\label{eq2.2:eq3}
\phi(z)=1+c_1z+c_2z^2+\cdots\quad (z \in \U)
\end{equation}
such that $\text{Re}\{\phi(z)\} > \gamma$ in $\U.$ We have

\begin{lemma}\cite{DZ}.\label{lem2} \;If the function $\phi$, given by \eqref{eq2.2:eq3}
 belongs to the class $\mathcal P(\gamma)$, then
 \begin{equation*}
 \text{Re}\{\phi(z)\} \geq 2\gamma-1+\dfrac{2(1-\gamma)}{1+|z|}\quad(0
 \leq \gamma < 1; z \in \mathbb U).
 \end{equation*}
  \end{lemma}

 \begin{lemma} \cite{MM3}\label{lem3}
 .\;If $-1 \leq B < A \leq 1, \beta^{*} > 0$ and the complex number
 $\gamma^{*}$ is constrained by $\text{Re}(\gamma^{*}) \geq
 -\beta^{*}(1-A)/(1-B)$, then the following differential equation:
  \begin{equation*}
 q(z)+\dfrac{zq^{\prime}(z)}{\beta^{*}q(z)+\gamma^{*}}=\dfrac{1+Az}{1+Bz}\quad
 (z \in \mathbb U)
 \end{equation*}
  has a univalent solution in $\mathbb U$ given by
 \begin{equation}\label{eq2.2:eq4}
 q(z)=\begin{cases}
 \dfrac{z^{\beta^{*}+\gamma^{*}}(1+Bz)^{\beta^{*}(A-B)/B}}{\beta^{*}\int_{0}^{z}t^{\beta^{*}
 +\gamma^{*}-1}(1+Bt)^{\beta^{*}(A-B)/B}\,dt}-\dfrac{\gamma^{*}}{\beta^{*}}, &  B \ne 0\\
\dfrac{z^{\beta^{*}+\gamma^{*}}\exp(\beta^{*}Az)}{\beta^{*}\int_{0}^{z}t^{\beta^{*}
 +\gamma^{*}-1}\exp(\beta^{*}At)dt}-\dfrac{\gamma^{*}}{\beta^{*}}, &
 B=0.
\end{cases}
\end{equation}
\end{lemma}

If the function $\phi$, given by \eqref{eq2.2:eq3} is analytic in
$\mathbb U$ and satisfies the following subordination:

\begin{equation}\label{eq2.2:eq5}
\phi(z)+\dfrac{z\phi^{\prime}(z)}{\beta^{*}\phi(z)+\gamma^{*}} \prec
\dfrac{1+Az}{1+Bz}\quad (z \in \mathbb U),
\end{equation}
then
\begin{equation*}
\phi(z) \prec q(z) \prec \dfrac{1+Az}{1+Bz}\quad (z \in \mathbb U)
\end{equation*}
and $q$ is the best dominant of \eqref{eq2.2:eq5}.

\begin{lemma} \cite{DRW} \label{lem4}.\;Let $\nu$ be a positive measure on $[0,1]$.
Let $h(z,t)$ be a complex-valued function defined on $\mathbb U
\times [0,1]$ such that $h(\cdot, t)$ is analytic in $\mathbb U$ for
each $t \in [0,1]$ and that $h(\cdot, t)$ is $\nu$-integrable on
$[0,1]$ for all $z \in \mathbb U$. In addition, suppose that
$\text{Re}\left\{h(\cdot, t)\right\} > 0, h(-r,t)$ is real and
\begin{equation*}
\text{Re}\left\{\dfrac{1}{h(z,t)}\right\}
\geq\dfrac{1}{h(-r,t)}\quad (|z| \leq r < 1, t \in [0,1]).
\end{equation*}
If the function $\mathcal H$ is defined in $\mathbb U$ by
\begin{equation*}
\mathcal H(z)=\int_0^1 h(z,t)\,d\nu(t),
\end{equation*}
then
\begin{equation*}
\text{Re}\left\{\dfrac{1}{\mathcal H(z)}\right\} \geq
\dfrac{1}{\mathcal H(-r)}.
\end{equation*}
\end{lemma}

 Each of the identities given below are well-known (see Whittaker
et al.~\cite[Chapter 14]{WW}) for the hypergeometric function
${_2}F_{1}.$

\begin{lemma}\label{lem5} \cite{WW}
For real or complex numbers $a,\;b$ and $c\;(c \ne 0, -1, -2,
\ldots),$ we have
\begin{align} \label{eq2.2:eq7}
&\int_0^1t^{b-1}(1-t)^{c-b-1}(1-tz)^{-a}\,dt=\dfrac{\Gamma(b)\Gamma(c-b)}{\Gamma(c)}\;_2F_1(a,b;c;z)\quad
(\text{Re}(c) > \text{Re}(b) > 0);\\
&_2F_1(a,b;c;z)=(1-z)^{-a}\;_2F_1\left(a,c-b;c;\frac{z}{z-1}\right);\label{eq2.2:eq8}\\
& _2F_1(a, b; c; z) =\;_2F_1(b, a; c; z);\label{eq2.2:eq9}\\
&(a+1)\;_2F_1(1, a; a+1; z)=(a+1)+az\;_2F_1(1, a+1;
a+2;z).\label{eq2.2:eq10}
\end{align}
\end{lemma}

\section{Inclusion relationships}

 Unless otherwise mentioned, we shall assume throughout the
sequel that $\alpha \in \mathbb R,\;\beta>0,$\\$\alpha+p\beta > 0,\;
\mu > 0,\;m \in \mathbb Z,\; -1 \leq B < A \leq 1$ and the powers
are understood as principal values.

In this section, we establish some inclusion relationships involving
the class\\ $S_{p,n}^{m}(\alpha,\beta,\mu,A,B)$.

\begin{theorem}\label{theorem 2.3.1}
If $f \in \mathcal S_{p}^{m}(\alpha, \beta, \mu;A, B)$, then
\begin{equation}\label{2.3.1}
\dfrac{\Theta_p^{m+1}(\alpha,\beta)f(z)}{\Theta_p^{m}(\alpha,\beta)f(z)}
\prec \dfrac{\mu\beta}{(\alpha+p\beta)Q(z)}=q(z) \prec
\dfrac{1+Az}{1+Bz}\quad (z \in \mathbb U),
\end{equation}
where
\begin{equation}\label{2.3.2}
Q(z)=
\begin{cases}\displaystyle \int_0^1 t^{\frac{\alpha+p\beta}{\mu
\beta}-1}\left(\dfrac{1+Btz}{1+Bz}\right)^{\frac{(\alpha+p\beta)(A-B)}{\mu
\beta\,B}}\,dt, & B \ne 0\\
\displaystyle\int_0^1 t^{\frac{\alpha+p\beta}{\mu
\beta}-1}\exp\left(\frac{(\alpha+p\beta)}{\mu \beta}A(t-1)z\right)dt, &
B=0
\end{cases}
\end{equation}
and $q$ is the best dominant of \eqref{2.3.1}. Furthermore, if
\begin{equation*}
A \leq -\dfrac{\mu \beta}{\alpha+p\beta}B \quad\text{with}\quad -1
\leq B < 0,
\end{equation*}
then
\begin{equation}\label{2.3.3}
\mathcal S_{p}^{m}(\alpha, \beta, \mu, A, B) \subset \mathcal
S_{p}^{m}(\alpha, \beta, \rho),
\end{equation}
where
\begin{equation*}
\rho=\left\{_2F_1\left(1,\dfrac{\alpha+p\beta}{\mu
\beta}\left(\dfrac{B-A}{B}\right);\dfrac{\alpha+p\beta}{\mu
\beta}+1;\dfrac{B}{B-1}\right)\right\}^{-1}.
\end{equation*}
The result is the best possible.
\end{theorem}

\begin{proof}\;
 Let $f \in \mathcal S_{p}^{m}(\alpha, \beta, \mu; A,B)$. Consider the function $g$ defined by
 \begin{equation}\label{2.3.4}
g(z)=z\left(\dfrac{\Theta_p^{m}(\alpha,\beta)f(z)}{z^p}\right)^{\frac{\beta}{\alpha+p\beta}}
\end{equation}
and let $r_1=\sup\{r: g(z) \ne 0, 0 < |z| \leq r < 1\}$. Then $g$ is
single-valued and analytic in $|z|< r_1.$ Taking logarithmic
differentiation in \eqref{2.3.4} and using the identity
\eqref{eq2:eq6} in the resulting equation, it follows that the
function $\phi$ given by
\begin{equation}\label{2.3.5}
\phi(z)=\dfrac{zg^{\prime}(z)}{g(z)}=\dfrac{\Theta_p^{m+1}(\alpha,\beta)f(z)}{\Theta_p^{m}(\alpha,\beta)f(z)}
\end{equation}
is analytic in $|z|< r_1$ and $\phi(0)=1.$ Carrying out logarithmic
differentiation in \eqref{2.3.5}, followed by the use of the
identity \eqref{eq2:eq6} and \eqref{eq2:eq7}, we deduce that
\begin{equation}\label{2.3.6}
\phi(z)+\dfrac{\mu \beta\,z
\phi^{\prime}(z)}{(\alpha+p\beta)\phi(z)} \prec
\dfrac{1+Az}{1+Bz}\quad (|z|< r_1).
\end{equation}
Hence, by using Lemma \ref{lem3}, we find that
\begin{equation}\label{2.3.7}
\phi(z) \prec \dfrac{\mu \beta}{(\alpha+p\beta)Q(z)}=q(z) \prec
\dfrac{1+Az}{1+Bz}\quad (|z|< r_1),
\end{equation}
where $q$ is the best dominant of \eqref{2.3.1} and is given by
\eqref{eq2.2:eq4} with $\beta^{*}=(\alpha+p\beta)/\mu \beta$ and
$\gamma^{*}=0$. For $-1 \leq B < A \leq 1$, it is easy to see that
\begin{equation*}
\text{Re}\left(\dfrac{1+Az}{1+Bz}\right) > 0\quad (z \in \mathbb U),
\end{equation*}
so that by \eqref{2.3.7}, we have
\begin{equation*}
\text{Re}\{\phi(z)\} > 0 \quad (|z|< r_1).
\end{equation*}
Now, \eqref{2.3.5} shows that the function $g$ is starlike
(univalent) in $|z|< r_1$. Thus, it is not possible that $g$
vanishes on $|z|=r_1$, if $r_1 < 1$. So, we conclude that $r_1=1$
and the function $\phi$ given by \eqref{2.3.5} is analytic in
$\mathbb U$. Hence, in view of \eqref{2.3.7}, we have
\begin{equation*}
\phi(z) \prec q(z) \prec \dfrac{1+Az}{1+Bz}\quad (z \in \mathbb U).
\end{equation*}
This proves the assertion \eqref{2.3.1}. To prove \eqref{2.3.3}, we
need to show that
\begin{equation}\label{2.3.8}
\inf_{z \in \mathbb U}\left\{\text{Re}(q(z))\right\}=q(-1).
\end{equation}
If we set
\begin{equation*}
a=\dfrac{\alpha+p\beta}{\mu \beta}\left(\dfrac{B-A}{B}\right),
b=\dfrac{\alpha+p\beta}{\mu
\beta}\;\;\text{and}\;\;c=\dfrac{\alpha+p\beta}{\mu \beta}+1,
\end{equation*}
then $c > b> 0$. From \eqref{2.3.2}, by using \eqref{eq2.2:eq7} to
\eqref{eq2.2:eq10}, we see that for $B \ne 0$,
\begin{align}\label{2.3.9}
Q(z)&=(1+Bz)^{a}\int_0^1 t^{b-1}(1+Btz)^{-a}\,dt\notag\\
 &
=\dfrac{\mu
\beta}{\alpha+p\beta}\,_2F_1\left(1,\dfrac{\alpha+p\beta}{\mu
\beta}\left(\dfrac{B-A}{B}\right);\dfrac{\alpha+p\beta}{\mu
\beta}+1;\dfrac{Bz}{1+Bz}\right).
\end{align}
To prove \eqref{2.3.8}, it suffices to show that
\begin{equation*}
\text{Re}\left(\dfrac{1}{Q(z)}\right) \geq \dfrac{1}{Q(-1)}\quad (z
\in \mathbb U).
\end{equation*}
Since
\begin{equation*}
A \leq -\dfrac{\mu \beta}{\alpha+p\beta}B\quad\text{with}\quad -1
\leq B < 0
\end{equation*}
implies that $c > a > 0$, by using \eqref{eq2.2:eq4}, we find from
\eqref{2.3.9} that
\begin{equation*}
Q(z)=\int_0^1 h(z,t)\,dt,
\end{equation*}
where
\begin{equation*}
h(z,t)=\dfrac{1+Bz}{1+B(1-t)z}\quad (0 \leq t \leq 1)\quad
\text{and}\quad
d\nu(t)=\dfrac{\Gamma(b)}{\Gamma(a)\Gamma(c-a)}t^{a-1}t^{c-a-1}\,dt
\end{equation*}
which is a positive measure on $[0,1]$. For $-1 \leq B < 0$, it may
be noted that $\text{Re}\{h(z,t)\} > 0$ and $h(-r,t)$ is real for $0
\leq |z| \leq r < 1$ and $t \in [0,1]$. Therefore, by using Lemma
\ref{lem4}, we obtain
\begin{equation*}
\text{Re}\left(\dfrac{1}{Q(z)}\right) \geq \dfrac{1}{Q(-r)}\quad
(|z|\leq r < 1)
\end{equation*}
which, upon letting $r \to 1^{-}$ yields
\begin{equation*}
\text{Re}\left(\dfrac{1}{Q(z)}\right) \geq \dfrac{1}{Q(-1)}.
\end{equation*}
Further, by taking $A \to
\left(-\mu\beta\,B/(\alpha+p\beta)\right)^{+}$ for the case
$A=-\mu\beta\,B/(\alpha+p\beta)$ and using \eqref{2.3.1}, we get
\eqref{2.3.3}.
 The result is the best possible as the function $q$ is the best
dominant of the subordination \eqref{2.3.1}. This evidently
completes the proof of Theorem \ref{theorem 2.3.1}.
\end{proof}

Letting $\mu=1$ in Theorem \ref{theorem 2.3.1} , we obtain the
following result which, in turn yields the corresponding work of
Patel et al. \cite[Corollary 1]{patel1} for $m=\alpha=0$.

\begin{corollary}\label{cor 2.3.1}
 \;If $A \leq -\beta B/(\alpha+p\beta)(-1 \leq B <
 0)$, then
 \begin{equation*}
\mathcal S_p^{m-1}(\alpha, \beta, A, B) \subset \mathcal
S_p^{m}(\alpha, \beta,\widetilde{\rho}),
\end{equation*}
where
\begin{equation*}
\widetilde{\rho}=\left\{_2F_1\left(1,\dfrac{\alpha+p\beta}{\beta}\left(\dfrac{B-A}{B}\right);\dfrac{\alpha+p\beta}{\beta}+1;\dfrac{B}{B-1}\right)\right\}^{-1}.
\end{equation*}
The result is the best possible.
\end{corollary}

Setting $m=-1,\; \mu=\beta=1$ and
$\widetilde{A}=\{(p+\alpha)A-\alpha B\}/p$ in Theorem \ref{theorem
2.3.1} , we deduce the following result obtained earlier by Patel et
al. \cite[Corollary 2]{patel1}.

\begin{corollary}\label{cor 2.3.2}
If $\alpha > -p, -1 \leq B < 0$ and
\begin{equation*}
\widetilde{A} \leq \min\left\{1+\dfrac{\alpha(1-B)}{p},
-\dfrac{(\alpha+1)B}{p}\right\}\leq 1,
\end{equation*}
then for $f \in \mathcal S_p^{\ast}(\widetilde{A}, B)$, we have
\begin{equation*}
\text{Re}\left(\dfrac{z^{\alpha}f(z)}{\int_0^1
t^{\alpha-1}f(t)dt}\right) >
(p+\alpha)\left\{_2F_1\left(1,p\left(\dfrac{B-\widetilde{A}}{B}\right);\alpha+p+1;\dfrac{B}{B-1}\right)\right\}^{-1}.
\end{equation*}
The result is the best possible.
\end{corollary}

Setting $\alpha=p+\ell-p\lambda$ and $\beta=\lambda$ in Theorem
\ref{theorem 2.3.1}, we get

\begin{corollary}\label{cor 2.3.3}
If $p>-\ell, -1 \leq B < 0, A \leq \mu \beta/(p+\ell)$ and $f \in
\mathcal S_p^m(\lambda, \ell, \mu,A,B)$, then
\begin{equation*}
\text{Re}\left(\dfrac{\mathcal J_p^{m+1}(\lambda,
\ell)f(z)}{\mathcal J_p^{m}(\lambda, \ell)f(z)}\right)>
\left\{_2F_1\left(1,\dfrac{(p+\ell)(B-A)}{\mu \lambda
B};\dfrac{p+\ell}{\mu \lambda}+1;\dfrac{B}{B-1}\right)\right\}^{-1}.
\end{equation*}
The result is the best possible.
\end{corollary}

\begin{theorem}\label{theorem 2.3.2}
If $\kappa=(1-A)/(1-B)$ and $f \in \mathcal S_p^{m}(\alpha,
\beta;\kappa)$,
 then
 \begin{equation*}
f \in \mathcal S_p^{m}(\alpha, \beta,
\mu,1-2\kappa,-1)\;\;\text{for}\;\;|z| < R=R(p, \alpha, \beta, \mu,
\kappa),
\end{equation*}
 where
 \begin{align}\label{2.3.10}
R=\begin{cases}\dfrac{(\alpha+p\beta)(1-\kappa)+\mu
\beta-\sqrt{\{(\alpha+p\beta)\kappa-\mu\beta\}^2+2\mu
\beta(\alpha+p\beta)}}{(\alpha+p\beta)(1-2\kappa)},
& \kappa\ne \dfrac{1}{2}\\
\dfrac{\alpha+p\beta}{2\{(\alpha+p\beta)(1-\kappa)+\mu \beta\}}, &
 \kappa=\dfrac{1}{2}.
\end{cases}
\end{align}
The result is the best possible.
\end{theorem}

\begin{proof}
Since $f \in \mathcal S_p^{m}(\alpha, \beta,\kappa)$, we have
\begin{equation}\label{2.3.11}
\dfrac{\Theta_p^{m+1}(\alpha,\beta)f(z)}{\Theta_p^{m}(\alpha,\beta)f(z)}=\kappa+(1-\kappa)\phi(z)\;(z\in\mathbb
U)
\end{equation}
where $\phi$, given by \eqref{eq2.2:eq3} is analytic and has a
positive real part in $\mathbb U$. Taking logarithmic
differentiation in \eqref{2.3.11}, and using \eqref{eq2:eq6} in the
resulting equation followed by simplifications, we deduce that
\begin{align}\label{2.3.12}
\text{Re}&\left\{(1-\mu)\dfrac{\Theta_{p}^{m+1}(\alpha,\beta)f(z)}{\Theta_{p}^{m}(\alpha,\beta)f(z)}
+\mu\dfrac{\Theta_{p}^{m+2}(\alpha,\beta)f(z)}{\Theta_{p}^{m+1}(\alpha,\beta)f(z)}\right\}-\kappa\notag\\
&\geq (1-\kappa)\left\{\text{Re}(\phi(z))-\dfrac{\mu
\beta}{\alpha+p\beta}\dfrac{|z\phi^{\prime}(z)|}{|\kappa+(1-\kappa)\phi(z)|}\right\}.
\end{align}
Now, by using the well-known \cite{mag1} estimates
\begin{equation}\label{2.3.13}
\text{(i)}\; |z\phi^{\prime}(z)| \leq
\dfrac{2nr^n}{1-r^{2n}}\text{Re}(\phi(z))\quad \text{and}\quad
\text{(ii)}\;\text{Re}\{\phi(z)\} \geq \dfrac{1-r^n}{1+r^n}\quad
(|z|=r < 1)
\end{equation}
with $n=1$ in \eqref{2.3.12}, we obtain
\begin{align*}
\text{Re}&\left\{(1-\mu)\dfrac{\Theta_{p}^{m+1}(\alpha,\beta)f(z)}{\Theta_{p}^{m}(\alpha,\beta)f(z)}
+\mu\dfrac{\Theta_{p}^{m+2}(\alpha,\beta)f(z)}{\Theta_{p}^{m+1}(\alpha,\beta)f(z)}\right\}-\kappa\notag\\
&\geq (1-\kappa)\text{Re}\{\phi(z)\}\left[1-\dfrac{2\mu\beta
r}{(\alpha+p\beta)\left\{\kappa(1-r^2)+(1-\kappa)(1-r)^2\right\}}\right]
\end{align*}
which is certainly positive, if $r < R$, where $R$ is given by
\eqref{2.3.10}.

  It is easily seen that the bound $R$ is the best possible for the
function $f \in \mathcal A_p$ defined by

\begin{equation*}
\dfrac{\Theta_{p}^{m+1}(\alpha,\beta)f(z)}{\Theta_{p}^{m}(\alpha,\beta)f(z)}=\dfrac{1+(1-2\kappa)z}{1-z}\quad
\left(\kappa=\dfrac{1-A}{1-B}; z \in \mathbb U\right).
\end{equation*}
\end{proof}

  A special case of Theorem \ref{theorem 2.3.2} when $m=\alpha=0,\; A=1-(2\rho/p)$ and
$B=-1$ we have.

{\bf Corollary 2.3.4.}\;If $0 \leq \rho < p$ and $f \in \mathcal
S_p^{\ast}(\rho)$, then
\begin{equation*}
\text{Re}\left\{(1-\mu)\dfrac{zf^{\prime}(z)}{f(z)}+\mu\left(1+\dfrac{zf^{\prime\prime}(z)}{f^{\prime}(z)}\right)\right\}>
\rho\quad \text{for}\quad |z| < R(p,\mu,\rho),
\end{equation*}
where
\begin{equation*}
R(p,\mu,\rho)=\begin{cases}\dfrac{(p+\mu-\rho)-\sqrt{\rho^2+2\mu(p-\rho)+\mu^2}}{p-2\rho},
& \rho \ne \dfrac{p}{2}\\ \dfrac{p}{2(p+\mu-\rho)}, &
\rho=\dfrac{p}{2}.
\end{cases}
\end{equation*}
The result is the best possible.

  For a function $f \in \mathcal A_p(n)$, we define the integral
operator $\mathcal F_{\delta,p} : \mathcal A_p(n) \longrightarrow
\mathcal A_p(n)$  by

\begin{equation}\label{2.3.14}
 \mathcal F_{\delta,p}(f)(z)=\dfrac{\delta+p}{z^{\delta}}\int_0^z
t^{\delta-1}f(t)dt \quad (\delta > -p; z \in \mathbb U).
\end{equation}
If $f$ is defined by \eqref{eq2:eq1}, then
\begin{align}\label{2.3.15}
\mathcal F_{\delta,p}(f)(z)=& z^p+\sum_{k=n}^
{\infty}\dfrac{\delta+p}{\delta+p+k}a_{p+k}z^{p+k}\quad (z \in \mathbb U) \\
&=z^{p}\; _2F_1(1,\delta+p;\delta+p+1;z) \star f(z)
=\Theta_p^{-1}(n,\delta, 1)f(z)\notag\quad (z \in \mathbb U).
\end{align}

It follows from \eqref{eq2:eq5} and \eqref{2.3.15} that for $f \in
\mathcal A_p(n)$ and $\delta > -p$,

\begin{equation}\label{2.3.16}
z\left(\Theta_p^{m}(n;\alpha, \beta)\mathcal
F_{\delta,p}(f)(z)\right)^{\prime}=(\delta+p)\Theta_p^{m}(n;\alpha,
\beta)f(z)-\delta \Theta_p^{m}(n;\alpha, \beta)\mathcal
F_{\delta,p}(f)(z)\;(z \in \mathbb U).
\end{equation}

Now we have

\begin{theorem}\label{theorem 2.3.3}
 Let $\delta$ be a real number satisfying the condition
\begin{equation*}
\delta \geq \dfrac{\alpha(A-B)-p\beta(1-A)}{\beta(1-B)}.
\end{equation*}
(i)\;If $f \in \mathcal S_p^{m}(\alpha, \beta, A, B)$, then the
function $\mathcal F_{\delta,p}(f)$ given by \eqref{2.3.14} belongs
to the class $\mathcal S_p^{m}(\alpha, \beta, A, B)$. Furthermore,
\begin{equation*}
\dfrac{\Theta_p^{m+1}(\alpha, \beta)\mathcal
F_{\delta,p}(f)(z)}{\Theta_p^{m}(\alpha, \beta)\mathcal
F_{\delta,p}(f)(z)} \prec
\dfrac{1}{\alpha+p\beta}\left(\dfrac{\beta}{Q(z)}-(\delta\beta-\alpha)\right)=q(z)
\prec \dfrac{1+Az}{1+Bz} \quad (z \in \mathbb U),
\end{equation*}
where
\begin{equation}\label{2.3.17}
Q(z)=\begin{cases}\displaystyle \int_0^1
t^{\delta+p-1}\left(\dfrac{1+Btz}{1+Bz}\right)^{\frac{(\alpha+p\beta)(A-B)}{\beta\,B}}\,dt,
& B \ne 0 \\ \displaystyle\int_0^1
t^{\delta+p-1}\exp\left(\frac{(\alpha+p\beta)}{\beta}A(t-1)\right),
& B=0
\end{cases}
\end{equation}
and $q$ is the best dominant.\\
(ii)\;If $-1 \leq B < 0$ and
\begin{equation*}
\delta \geq
\max\left\{\dfrac{\alpha+p\beta}{\beta}\left(\dfrac{B-A}{B}\right)-p-1,
\dfrac{\alpha}{\beta}-\dfrac{(\alpha+p\beta)(1-A)}{\beta(1-B)}
\right\}.
\end{equation*}
 then
\begin{equation*}
f \in \mathcal S_p^{m}(\alpha, \beta, A, B)  \Longrightarrow
\mathcal F_{\delta,p}(f) \in \mathcal S_p^{m}(\alpha, \beta; \tau),
\end{equation*}
where
\begin{equation*}
\tau=\dfrac{1}{\alpha+p\beta}\left[\beta(\delta+p)
\left\{_2F_1\left(1,\dfrac{(\alpha+p\beta)(B-A)}{\beta\,B};\delta+p+1;\dfrac{B}{B-1}\right)\right\}^{-1}-(\delta\beta-\alpha)\right].
\end{equation*}
 The result is the best possible.
 \end{theorem}

\begin{proof}\;Setting
\begin{equation}\label{2.3.18}
g(z)=z\left(\dfrac{\Theta_p^{m}(\alpha, \beta)\mathcal
F_{\delta,p}(f)(z)}{z^p}\right)^{\frac{\beta}{\alpha+p\beta}}
\end{equation}
and $r_1=\sup\{r : g(z) \ne 0, 0 < |z|\leq r < 1\}$, we see that $g$ is
single-valued and analytic in $|z|< r_1$. By taking the logarithmic
differentiation in \eqref{2.3.18} and using the identity
\eqref{2.3.16} for the function $\mathcal F_{\delta,p}(f)$, it
follows that

\begin{equation}\label{2.3.19}
\phi(z)=\dfrac{zg^{\prime}(z)}{g(z)}=\dfrac{\Theta_p^{m+1}(\alpha,
\beta)\mathcal F_{\delta,p}(f)(z)}{\Theta_p^{m}(\alpha,
\beta)\mathcal F_{\delta,p}(f)(z)}
\end{equation}

is analytic in $|z| < r_1$ and $\phi(0)=1$. Again, by making use of
the identity \eqref{eq2:eq6} and \eqref{2.3.16}, we deduce that
\begin{equation}\label{2.3.20}
\dfrac{\Theta_p^{m}(\alpha, \beta)(f)(z)}{\Theta_p^{m}(\alpha,
\beta)\mathcal
F_{\delta,p}(f)(z)}=\dfrac{(\alpha+p\beta)\phi(z)+(\delta
\beta-\alpha)}{\beta(\delta+p)}\quad (|z| < r_1).
\end{equation}
Since $f \in \mathcal S_p^{m}(\alpha, \beta, A, B)$, it is clear
that $\Theta_p^{m}(\alpha, \beta)(f)(z) \ne 0$ in $0 < |z| < 1$. So,
in view of \eqref{2.3.20}, we have
\begin{equation}\label{2.3.21}
\dfrac{\Theta_p^{m}(\alpha, \beta)\mathcal
F_{\delta,p}(f)(z)}{\Theta_p^{m}(\alpha,
\beta)(f)(z)}=\dfrac{\beta(\delta+p)}{(\alpha+p\beta)\phi(z)+(\delta
\beta-\alpha)}\quad (|z| < r_1).
\end{equation}
Now, by carrying out logarithmic differentiation in both sides of
\eqref{2.3.21} followed by the use of the identity \eqref{eq2:eq6},
\eqref{2.3.16} and \eqref{2.3.19} in the resulting equation, we
obtain
\begin{equation}\label{2.3.22}
\dfrac{\Theta_p^{m+1}(\alpha, \beta)(f)(z)}{\Theta_p^{m}(\alpha,
\beta)(f)(z)}=\phi(z)+\dfrac{z\phi^{\prime}(z)}{\left(\dfrac{\alpha+p\beta}{\beta}\right)\phi(z)+\left(\delta-\dfrac{\alpha}{\beta}\right)}
\prec \dfrac{1+Az}{1+Bz}\quad (|z| < r_1).
\end{equation}
Thus, by making use of  Lemma \ref{lem3}  with
$\beta^{*}=(\alpha+p\beta)/\beta$ and $\gamma^{*}=(\delta
\beta-\alpha)/\beta$ in \eqref{2.3.22}, we get
\begin{equation}\label{2.3.23}
\phi(z)=\dfrac{1}{\alpha+p\beta}\left(\dfrac{\beta}{Q(z)}-(\delta\beta-\alpha)\right)=q(z)
\prec \dfrac{1+Az}{1+Bz}\quad (|z| < r_1),
\end{equation}
where $Q$ is given by \eqref{2.3.17}, and $q$ is the best dominant.\\
  Since for $-1 \leq B < A \leq 1$,
\begin{equation*}
\text{Re}\left(\dfrac{1+Az}{1+Bz}\right) > 0 \quad (z \in \mathbb
U),
\end{equation*}
by \eqref{2.3.22}, we have $\text{Re}\{\phi(z)\} > 0$ in $|z| <
r_1$.  Now, in view of \eqref{2.3.19} the function $g$ is univalent
in $|z|< r_1$. Thus, it is not possible that the function $g$
vanishes on $|z|=r_1$, if $r_1 < 1$. So, we conclude that $r_1=1$
and the function $\phi$ is analytic in $\mathbb U$. From
\eqref{2.3.19} and \eqref{2.3.23}, we prove the assertion (i) of
Theorem \ref{theorem 2.3.3} .

Following the same technique as in the proof of Theorem \ref{theorem
2.3.1}, we can prove the assertion (ii) of Theorem \ref{theorem
2.3.3}. The result is the best possible as $q$ is the best dominant.
\end{proof}

\begin{remark}\label{remark 2.3.1}
If, in Theorem \ref{theorem 2.3.3} with $A=1-(2\eta/p)\,\;(0 \leq
\eta < p),\; B=-1$, we set $m=\alpha=0$ and $m=1,\; \alpha=0$, we
shall obtain the corresponding results by Patel et al. \cite[Remark
2]{patel2}.
\end{remark}

\section{Properties involving the operator $\Theta_p^{m}(\alpha,\;\beta)$}

In this section, we derive certain properties and characteristics of
functions in $\mathcal A_p$ involving operator
$\Theta_p^{m}(\alpha,\beta)f(z)$

\begin{theorem}\label{theorem 2.4.1}
 Let $0 < \mu < 1$, $0 < \gamma \leq 1$ and $A
\leq 1-\mu(1-B)$. If $f \in \mathcal A_p$ satisfies the following
subordination condition
\begin{equation}\label{2.4.1}
 (1-\mu)\left(\dfrac{\Theta_p^{m}(\alpha,\;\beta)f(z)}{z^p}\right)^{\frac{1}{\gamma}}+\mu\dfrac{\Theta_p^{m+1}(\alpha, \beta)f(z)}{\Theta_p^{m}(\alpha,\;
\beta)f(z)}\prec \dfrac{1+Az}{1+Bz}\; (z \in \mathbb U),
\end{equation}
then
\begin{equation}\label{2.4.2}
\left(\dfrac{\Theta_p^{m}(\alpha,\;\beta)f(z)}{z^p}\right)^{\frac{1}{\gamma}}
\prec \dfrac{\mu \beta
\gamma}{(1-\mu)(\alpha+p\beta)}\left(\dfrac{1}{Q(z)}\right)=q(z)
\prec \dfrac{1+\frac{A-\mu B}{1-\mu}z}{1+Bz}\;\; (z \in \mathbb U),
\end{equation}
where
\begin{equation*}
Q(z)=\begin{cases} \displaystyle\int_0^1
t^{\frac{(1-\mu)(\alpha+p\beta)}{\mu \beta
\gamma}-1}\left(\dfrac{1+Btz}{1+Bz}\right)^{\frac{\alpha+p\beta}{\mu\beta\gamma}\left(\frac{A-B}{B}\right)}\,dt,
& B\ne 0\\
\displaystyle\int_0^1 t^{\frac{(1-\mu)(\alpha+p\beta)}{\mu \beta
\gamma}-1}\exp\left(\frac{\alpha+p\beta}{\mu \beta
\gamma}A(t-1)\right)\,dt, & B =0
\end{cases}
\end{equation*}
and $q$ is the best dominant of \eqref{2.4.2}. Furthermore, if
\begin{equation*}
A \leq
\min\left\{1-\mu(1-B),-\mu\left(\dfrac{\beta\gamma}{\alpha+p\beta}-1\right)B\right\}\quad\text{with}\quad
-1 \leq B < 0,
\end{equation*}
then
\begin{equation*}
\text{Re}\left(\dfrac{\Theta_p^{m}(\alpha,\;\beta)f(z)}{z^p}\right)^{\frac{1}{\gamma}}>
\xi\quad (z \in \mathbb U),
\end{equation*}
where
\begin{equation*}
\xi=\left\{_2F_1\left(1,\frac{\alpha+p\beta}{\mu\beta\gamma}\left(\frac{B-A}{B}\right);\frac{(1-\mu)(\alpha+p\beta)}{\mu\beta\gamma}+1;
\frac{B}{B-1}\right)\right\}^{-1}.
\end{equation*}
 The result is the best possible.
\end{theorem}

 \begin {proof} Setting
\begin{equation}\label{2.4.3}
\phi(z)=\left(\dfrac{\Theta_p^{m}(\alpha,
\beta)f(z)}{z^p}\right)^{\frac{1}{\gamma}}\;\;(\gamma>0,\;z \in
\mathbb U),
\end{equation}
we note that the function $\phi$ of the form \eqref{eq2.2:eq3} and
it is analytic in $\mathbb U$. Taking logarithmic differentiation in
both sides of \eqref{2.4.3} and using \eqref{eq2:eq6} in the
resulting equation, we deduce that
\begin{equation}\label{2.4.4}
\psi(z)+\dfrac{z\psi^{\prime}(z)}{\beta^{*}\psi(z)+\gamma^{*}} \prec
\dfrac{1+Az}{1+Bz}\quad(z \in \mathbb U),
\end{equation}
where $\psi(z)=\mu+(1-\mu)\phi(z),\;
\beta^{*}=(\alpha+p\beta)/\mu\beta\gamma$ and
$\gamma^{*}=-(\alpha+p\beta)/\beta\gamma$. Applying Lemma \ref{lem3}
in \eqref{2.4.4} and following the lines of proof of Theorem
\ref{theorem 2.3.1}, we shall obtain the assertion of Theorem
\ref{theorem 2.4.1}
\end{proof}

Letting $m=1,\;\alpha=0,\;\beta=1,\; A=1-(2\eta/p)$ and $B=-1$ in
Theorem \ref{theorem 2.4.1}, we get

\begin{corollary}\label{cor 2.4.1} If $\max\left\{p\mu,\;
\dfrac{p+(p-1)\mu}{2}\right\}\leq \eta< p$ and $f \in \mathcal A_p$
satisfies
\begin{equation*}
\text{Re}\left\{(1-\mu)\dfrac{f^{\prime}(z)}{z^{p-1}}+\mu\left(1+\dfrac{zf^{\prime\prime}(z)}{f^{\prime}(z)}\right)\right\}
> \eta\quad (z \in \mathbb U),
\end{equation*}
then
\begin{equation*}
\text{Re}\left\{\dfrac{f^{\prime}(z)}{z^{p-1}}\right\}>p\left\{
_2F_1\left(1,\dfrac{2(p-\eta)}{\mu};\dfrac{p(1-\mu)}{\mu}+1;
\dfrac{1}{2}\right)\right\}^{-1}\quad (z \in \mathbb U).
\end{equation*}
The result is the best possible.
\end{corollary}

Setting
\begin{equation*}
\left(\dfrac{\Theta_p^{m}(\alpha,
\beta)f(z)}{z^p}\right)^{\frac{1}{\gamma}}=\kappa+(1-\kappa)\phi(z)\quad
\left(0 < \gamma \leq 1, \kappa=\dfrac{1-A}{1-B};z \in \mathbb
U\right),
\end{equation*}
where $\phi$ is of the form \eqref{eq2.2:eq1}, using the estimates
\eqref{2.3.13}
 and following the lines of proof of Theorem \ref{theorem 2.3.2}, we obtain

 \begin{theorem}\label{theorem 2.4.2}
  Let $0 <\mu <1,\; 0 < \gamma \leq 1$ and $f
\in \mathcal A_p(n)$ satisfies the following subordination condition
\begin{equation*}
\dfrac{\Theta_p^{m}(n,\alpha,
 \beta)f(z)}{z^p} \prec
\left(\dfrac{1+Az}{1+Bz}\right)^{\gamma}\quad (z \in \mathbb U),
\end{equation*}
then
\begin{equation*}
\text{Re}\left\{(1-\mu)\left(\dfrac{\Theta_p^{m}(n;\alpha,
 \beta)f(z)}{z^p}\right)^{\frac{1}{\gamma}}+\mu
\dfrac{\Theta_p^{m+1}(n;\alpha, \beta)f(z)}{\Theta_p^{m}(n;\alpha,
\beta)f(z)}\right\}> \mu+(1-\mu)\kappa
\end{equation*}
for $|z|< R\equiv R(p,\;n,\;\mu,\;\alpha,\;\beta,\;\gamma,\;
\kappa),$ where $R$ is the smallest positive root of the equation
\begin{equation*}
(1-\mu)(\alpha+p\beta)(1-2\kappa)r^{2n}-2\{(1-\mu)(\alpha+p\beta)(1-\kappa)+n\mu\beta\gamma\}r^n+(1-\mu)(\alpha+p\beta)=0.
\end{equation*}
\end{theorem}
The result is the best possible for the function $f \in \mathcal
A_p(n)$ defined by
\begin{equation*}
\left(\dfrac{\Theta_p^{m}(n;\alpha,\;\beta)f(z)}{z^p}\right)^{\frac{1}{n}}=\left(\dfrac{1+(1-2\kappa)z^n}{1-z^n}
\right)^{\gamma}\; (0 < \gamma \leq 1, \kappa=\dfrac{1-A}{1-B}; z
\in \mathbb U).
\end{equation*}
Next, we derive the following result.

\begin{theorem}\label{theorem 2.4.3}
If $\mu > 0$ and $f \in \mathcal A_p(n)$ satisfies the following
subordination condition:
\begin{equation}\label{2.4.5}
(1-\mu)\dfrac{\Theta_p^{m}(n;\alpha,
 \beta)f(z)}{z^p}+\mu\dfrac{\Theta_p^{m+1}(n;\alpha, \beta)f(z)}{z^p}
\prec \dfrac{1+Az}{1+Bz}\quad (z \in \mathbb U),
\end{equation}
then for $z \in \mathbb U$,
\begin{align}\label{2.4.6}
&\dfrac{\Theta_p^{m}(n;\alpha, \beta)f(z)}{z^p}\notag\\
& \prec
\begin{cases}
\dfrac{A}{B}+\left(1-\dfrac{A}{B}\right)(1+Bz)^{-1}\,_2F_1\left(1,1;\dfrac{\alpha+p\beta}{\mu\beta
n}+1;\dfrac{Bz}{1+Bz}\right), & B \ne 0\\
1+\dfrac{\alpha+p\beta}{\alpha+p\beta+\mu\beta n}Az, & B=0,
\end{cases}.
\end{align}
Further,
\begin{equation}\label{2.4.7}
\text{Re}\left\{\left(\dfrac{\Theta_p^{m}(n;\alpha,
 \beta)f(z)}{z^p}\right)^{1/t}\right\} > \sigma^{1/t}\quad (t \in
\mathbb N; z \in \mathbb U),
\end{equation}
where
\begin{equation*}
\sigma=\begin{cases}
\dfrac{A}{B}+\left(1-\dfrac{A}{B}\right)(1-B)^{-1}\,_2F_1\left(1,1;\dfrac{\alpha+p\beta}{\mu\beta
n}+1;\dfrac{B}{B-1}\right), & B \ne 0\\
1-\dfrac{\alpha+p\beta}{\alpha+p\beta+\mu\beta n}A, & B=0,
\end{cases}
\end{equation*}
The result is the best possible.
\end{theorem}

\begin{proof}
For $f \in \mathcal A_p(n)$, we write
\begin{equation}\label{2.4.8}
\phi(z)=\dfrac{\Theta_p^{m}(n;\alpha, \beta)f(z)}{z^p}\quad (f \in
\mathcal A_p(n); z \in \mathbb U).
\end{equation}
Then, $\phi$ is of the form \eqref{eq2.2:eq1} and it is analytic in
the unit disk $\mathbb U$. On differentiating both the sides of
\eqref{2.4.8}, using the identity \eqref{eq2:eq6} in the resulting
equation followed by the use of \eqref{2.4.5}, we get
\begin{equation}\label{2.4.9}
\phi(z)+\dfrac{z\phi^{\prime}(z)}{(\alpha+p\beta)/\mu\beta} \prec
\dfrac{1+Az}{1+Bz}\quad (z \in \mathbb U).
\end{equation}
Now, by an application of Lemma \ref{lem1} (with
$\gamma=(\alpha+p\beta)/\mu\beta$) in \eqref{2.4.9}, we obtain
\begin{equation*}
\dfrac{\Theta_p^{m}(n;\alpha, \beta)f(z)}{z^p} \prec
Q(z)=\dfrac{\alpha+p\beta}{\mu\beta\, n}
z^{-\frac{\alpha+p\beta}{\mu\beta\,n}}\int_0^z
t^{\frac{\alpha+p\beta}{\mu\beta\,
n}-1}\,\left(\dfrac{1+At}{1+Bt}\right)\,dt\quad (z \in \mathbb U)
\end{equation*}
which yields \eqref{2.4.6} by change of variables followed by the
use of the identities \eqref{eq2.2:eq7} to \eqref{eq2.2:eq10} (with
$a=1, b=(\alpha+p\beta)/\mu\beta n$ and $c=b+1$). This proves the
assertion \eqref{2.4.6} of Theorem \ref{theorem 2.4.3}.

  To prove \eqref{2.4.7}, it suffices to show that
\begin{equation}\label{2.4.10}
\inf_{z \in \mathbb U}\left\{\text{Re}(Q(z))\right\}=Q(-1).
\end{equation}
Indeed, for $|z| \leq r < 1$,
\begin{equation*}
\text{Re}\left(\dfrac{1+Az}{1+Bz}\right) \geq \dfrac{1-Ar}{1-Br}.
\end{equation*}
Setting
\begin{equation*}
\mathcal G(s,z)=\dfrac{1+Asz}{1+Bsz}\quad (0 \leq s \leq 1)\quad
\text{and}\quad d\nu(s)=\dfrac{\alpha+p\beta}{\mu\beta
n}s^{\frac{\alpha+p\beta}{\mu\beta n}-1}ds,
\end{equation*}
which is a positive measure on $[0,1]$, we get
\begin{equation*}
Q(z)=\int_0^1\mathcal G(s,z)d\nu(s),
\end{equation*}
so that
\begin{equation*}
\text{Re}\{Q(z)\} \geq \int_0^1
\dfrac{1-Asz}{1-Bsz}d\nu(s)=Q(-r)\quad (|z| \leq r <1).
\end{equation*}
Upon letting $r \to 1^{-}$ in the above inequality, we obtain the
assertion \eqref{2.4.9}. Now, with the aid of the elementary
inequality:
\begin{equation*}
\text{Re}\left(\omega^{1/t}\right) \geq
\left(\text{Re}(\omega)\right)^{1/t}\quad (\text{Re}(\omega) > 0; t
\in \mathbb N),
\end{equation*}
the estimate \eqref{2.4.7} follows from \eqref{2.4.10}.

  The estimate in \eqref{2.4.7} is the
best possible as the function $Q$ is the best dominant of
\eqref{2.4.6}.
\end{proof}

Putting $\alpha=p+\ell-p\lambda, \beta=\lambda$ and $t=1$ in Theorem
\ref{theorem 2.4.3}, we get the following result.

\begin{corollary}\label{cor 2.4.2}
If $f \in \mathcal A_p$ satisfies
\begin{equation*}
(1-\mu)\dfrac{\mathcal J_p^{m}(\lambda, \ell)f(z)}{z^p}+\mu
\dfrac{\mathcal J_p^{m+1}(\lambda, \ell)f(z)}{z^p} \prec
\dfrac{1+Az}{1+Bz}\quad (z \in \mathbb U),
\end{equation*}
then \begin{equation*} \text{Re}\left\{\dfrac{\mathcal
J_p^{m}(\lambda, \ell)f(z)}{z^p}\right\} > \varrho\quad (z \in
\mathbb U),
\end{equation*}
where
\begin{equation*}
\varrho=\begin{cases}
\dfrac{A}{B}+\left(1-\dfrac{A}{B}\right)(1-B)^{-1}\,_2F_1\left(1,1;\dfrac{p+\ell}{\mu\lambda
n}+1;\dfrac{B}{B-1}\right), & B \ne 0\\
1-\dfrac{p+\ell}{p+\ell+\mu\lambda n}A, & B=0,
\end{cases}
\end{equation*}
The result is the best possible.
\end{corollary}

Setting $m=-1,\; \alpha=\delta,\; \beta=1,\; A=1-2\eta\,\;(0 \leq
\eta < 1)$ and $B=-1$ in Theorem \ref{theorem 2.4.3}, we obtain

\begin{corollary}\label{cor 2.4.3}
If $f \in \mathcal A_p(n)$ satisfies
\begin{equation*}
(1-\mu)\dfrac{\mathcal F_{\delta, p}(f)(z)}{z^p}+\mu
\dfrac{f(z)}{z^p} \prec \dfrac{1+Az}{1+Bz}\quad (\mu > 0, \delta >
-p; z \in \mathbb U),
\end{equation*}
then
\begin{equation*}
\text{Re}\left(\dfrac{\mathcal F_{\delta, p}(f)(z)}{z^p}\right) >
\eta+(1-\eta)\,\left\{_2F_1\left(1,1;\dfrac{\delta+p}{\mu
n}+1;\dfrac{1}{2}\right)-1\right\}\quad (z \in \mathbb U).
\end{equation*}
The result is the best possible.
\end{corollary}

\begin{theorem}\label{theorem 2.4.4}
If $\kappa=(1-A)/(1-B)$ and $f \in \mathcal A_p(n)$ satisfies the
subordination condition:
\begin{equation}\label{2.4.11}
\dfrac{\Theta_p^{m}(n;\alpha, \beta)f(z)}{z^p} \prec
\dfrac{1+Az}{1+Bz}\quad (z \in \mathbb U),
\end{equation}
then
\begin{align*}
\text{Re}\left\{(1-\mu)\dfrac{\Theta_p^{m}(n;\alpha,
\beta)f(z)}{z^p}+\mu\dfrac{\Theta_p^{m+1}(n;\alpha, \beta)f(z)}{z^p}
\right\} > \kappa &\\ \text{for}\; |z| < R=R(p, \alpha,\;
\beta,\;\mu,\; n),
\end{align*}
where
\begin{equation*}
R=\left[\dfrac{\sqrt{(\alpha+p\beta)^2+(\mu \beta n)^2}-\mu \beta
n}{\alpha+p\beta}\right]^{\frac{1}{n}}.
\end{equation*}
The result is the best possible.
\end{theorem}

\begin{proof}
From \eqref{2.4.11}, we note that
\begin{equation*}
\text{Re}\left\{\dfrac{\Theta_p^{m}(n;\alpha,
\beta)f(z)}{z^p}\right\}
>\kappa \quad (z \in \mathbb U)
\end{equation*}
so that
\begin{equation}\label{2.4.12}
\dfrac{\Theta_p^{m}(n;\alpha,
\beta)f(z)}{z^p}=\kappa+(1-\kappa)\phi(z)\quad (z \in \mathbb U),
\end{equation}
where $\phi$, given by \eqref{eq2.2:eq1} is analytic and has a
positive real part in $\mathbb U$. Taking logarithmic
differentiation in both sides of  \eqref{2.4.12}, and using
\eqref{eq2:eq6} in the resulting equation, we deduce that
\begin{align*}
\text{Re}&\left\{(1-\mu)\dfrac{\Theta_p^{m}(n;\alpha,
 \beta)f(z)}{z^p}+\mu\dfrac{\Theta_p^{m+1}(n;\alpha, \beta)f(z)}{z^p}
\right\}-\kappa\\
& \geq (1-\kappa)\left\{\text{Re}(\phi(z))-\dfrac{\mu \beta
n}{\alpha+p\beta}|z\phi^{\prime}(z)|\right\}\quad (z \in \mathbb U).
\end{align*}
By using the estimate (i) of \eqref{2.3.13} in the above inequality
and following the lines of proof of Theorem \ref{theorem 2.3.2}, we
get the required assertion of Theorem \ref{theorem 2.4.4}

 It is easily seen that the bound $R$ is the best possible for the
function $f \in \mathcal A_p(n)$ defined by
\begin{equation*}
\dfrac{\Theta_p^{m}(n;\alpha,
 \beta)f(z)}{z^p}=\dfrac{1+(1-2\kappa)z^n}{1-z^n}\quad
\left(\kappa=\dfrac{1-A}{1-B}; z \in \mathbb U\right).
\end{equation*}
\end{proof}

  Putting $A=1-2\eta, B=-1, m=-1, \alpha=\delta$ and $\beta=1$ in
Theorem \ref{theorem 2.4.4}, we get

\begin{corollary}\label{cor 2.4.4}
If $\mu > 0, \delta > -p$ and $f \in \mathcal A_p(n)$ satisfies
\begin{equation*}
\text{Re}\left(\dfrac{\mathcal F_{\delta, p}(f)(z)}{z^p}\right) >
\eta \quad (0 \leq \eta < 1; z \in \mathbb U),
\end{equation*}
then
\begin{equation*}
(1-\mu)\dfrac{\mathcal F_{\delta, p}(f)(z)}{z^p}+\mu
\dfrac{f(z)}{z^p} > \eta \quad \text{for}\quad |z| <
R(p,\delta,\mu,n),
\end{equation*}
where
\begin{equation*}
R(p,\delta,\mu,n)=\left[\dfrac{\sqrt{(\delta+p)^2+(\mu n)^2}-\mu
n}{\delta+p}\right]^{\frac{1}{n}}.
\end{equation*}
The bound $\widetilde{R}(p,\delta,\mu,n)$ is the best possible for
the function $f \in \mathcal A_p(n)$ defined by
\begin{equation*}
\dfrac{\mathcal F_{\delta, p}(f)(z)}{z^p}=\dfrac{1+(1-2\eta)
z^n}{1-z^n}\quad (0 \leq \eta < 1, \delta > -p; z \in \mathbb U).
\end{equation*}
\end{corollary}

\begin{theorem}\label{theorem 2.4.5}
Let $\mu >0$ and $\delta > -p$. Suppose that $f \in \mathcal A_p(n)$
and $\mathcal F_{\delta,p}(f)$ is given by \eqref{2.3.14}. If
\begin{equation}\label{2.4.13}
(1-\mu)\dfrac{\Theta_p^m(n;\alpha,\beta)\mathcal F_{\delta,
p}(f)(z)}{z^p}+\mu \dfrac{\Theta_p^m(n;\alpha,\beta)f(z)}{z^p} \prec
\dfrac{1+Az}{1+Bz}\quad (z \in \mathbb U),
\end{equation}
then
\begin{equation*}
\text{Re}\left\{\left(\dfrac{\Theta_p^m(n;\alpha,\beta)\mathcal
F_{\delta, p}(f)(z)}{z^p}\right)^{1/t}\right\} > \xi^{1/t}\quad (t
\in \mathbb N; z \in \mathbb U),
\end{equation*}
where
\begin{equation*}
\xi=\begin{cases}
\dfrac{A}{B}+\left(1-\dfrac{A}{B}\right)(1-B)^{-1}\,_2F_1\left(1,1;\dfrac{\delta+p}{\mu
n}+1;\dfrac{B}{B-1}\right), & B \ne 0\\
1-\dfrac{\delta+p}{\delta+p+\mu n}A, & B=0,
\end{cases}
\end{equation*}
The result is the best possible.
\end{theorem}

\begin{proof}
If, we let
\begin{equation}\label{2.4.14}
\phi(z)=\dfrac{\Theta_p^m(n;\alpha, \beta)\mathcal F_{\delta,
p}(f)(z)}{z^p}\quad (z \in \mathbb U),
\end{equation}
then $\phi$ is of the form \eqref{eq2.2:eq1} and it is analytic in
$\mathbb U$. On differentiating both the sides of \eqref{2.4.14} and
using \eqref{2.3.16} in conjunction with \eqref{2.4.13},we deduce
that
\begin{align*}
\phi(z)+\dfrac{z\phi^{\prime}(z)}{(\delta+p)/\mu}&=(1-\mu)\dfrac{\Theta_p^m(n;\alpha,\beta)\mathcal
F_{\delta, p}(f)(z)}{z^p}+\mu \dfrac{\Theta_p^m(n;\alpha,
\beta)f(z)}{z^p}\\& \prec \dfrac{1+Az}{1+Bz}\quad (z \in \mathbb U).
\end{align*}
The remaining part of the proof of Theorem \ref{theorem 2.4.5} is
similar to that of Theorem \ref{theorem 2.4.3} and we omit the
details.
\end{proof}

Letting $A=1-2\eta,\; B=-1,\; m=-1,\; \alpha=\delta$ and
$\mu=\beta=t=1$ in Theorem \ref{theorem 2.4.4}, we obtain

\begin{corollary}\label{cor 2.4.5}
Suppose that $f \in \mathcal A_p(n)$ and $\mathcal F_{\delta,p}(f)$
is given by \eqref{2.3.14}. If
\begin{equation*}
\text{Re}\left\{\dfrac{\mathcal F_{\delta, p}(f)(z)}{z^p}\right\}>
\eta \quad (0 \leq \eta < 1, \delta > -p;z \in \mathbb U),
\end{equation*}
then
\begin{equation*}
\text{Re}\left(\dfrac{\displaystyle\int_0^z t^{\delta-1}\mathcal
F_{\delta, p}(f)(t)dt}{z^{\delta+p}}\right)
>\dfrac{\eta+(1-\eta)\,\left\{_2F_1\left(1,1;\dfrac{\delta+p}{\mu
n}+1;\dfrac{1}{2}\right)-1\right\}}{\delta+p}\quad (z \in \mathbb
U).
\end{equation*}
The result is the best possible.
\end{corollary}

\begin{theorem}\label{theorem 2.4.6}
Let $\mu > 0$ and $-1 \leq B_j < A_j \leq 1\,(j=1,2)$. If the
functions $f_j \in \mathcal A_p$ satisfy the subordination condition
\eqref{2.4.5} and
\begin{equation*}
\mathcal H(z)=\Theta_p^m(\alpha,\beta)(f_1 \star f_2)(z)\quad (z \in
\mathbb U),
\end{equation*}
then
\begin{equation}\label{2.4.15}
\text{Re}\left\{(1-\mu)\dfrac{\Theta_p^m(\alpha,\beta)\mathcal
H(z)}{z^p}+\mu\dfrac{\Theta_p^{m+1}(\alpha,\beta)\mathcal
H(z)}{z^p}\right\}> \eta \quad (0 \leq \eta < 1; z \in \mathbb U),
\end{equation}
where
\begin{equation*}
\eta=1-\dfrac{4(A_1-B_1)(A_2-B_2)}{(1-B_1)(1-B_2)}\left\{1-\dfrac{1}{2}\;
_2F_1\left(1,1;\dfrac{\alpha+p\beta}{\mu
\beta}+1;\dfrac{1}{2}\right)\right\}.
\end{equation*}
The result is the best possible when $B_1=B_2=-1.$
\end{theorem}

\begin{proof}
On setting
\begin{equation}\label{2.4.16}
\phi_j(z)=(1-\mu)\dfrac{\Theta_p^m(\alpha,\beta)f_j(z)}{z^p}+\mu\dfrac{\Theta_p^{m+1}(\alpha,\beta)f_j(z)}{z^p}\quad
(j=1,2;z \in \mathbb U)
\end{equation}
and using \eqref{2.4.5}, we note that $\phi_j \in \mathcal
P(\gamma_j)$, where $\gamma_j=(1-A_j)/(1-B_j)$ for $j=1,2.$ Now, by
making use of the identity \eqref{eq2:eq6} and \eqref{2.4.16}, we
deduce that
\begin{equation}\label{2.4.17}
\Theta_p^m(\alpha,\beta)f_j(z)=\dfrac{p+\alpha \beta}{\mu
\beta}z^{p-\frac{p+\alpha\beta}{\mu \beta}}\int_0^z
t^{\frac{\alpha+p\beta}{\mu \beta}-1}\phi_j(t)dt\quad (j=1,2).
\end{equation}
Thus, by making use of \eqref{2.4.17} followed by  simple
calculations, we obtain
\begin{equation}\label{2.4.18}
\Theta_p^m(\alpha,\beta)\mathcal H(z)=\dfrac{p+\alpha \beta}{\mu
\beta}z^{p-\frac{p+\alpha\beta}{\mu \beta}}\int_0^z
t^{\frac{\alpha+p\beta}{\mu \beta}-1}\phi_0(t)dt\quad (z \in \mathbb
U),
\end{equation}
where
\begin{align}\label{2.4.19}
\phi_0(z)&=(1-\mu)\dfrac{\Theta_p^m(\alpha,\beta)\mathcal
H(z)}{z^p}+\mu\dfrac{\Theta_p^{m+1}(\alpha,\beta)\mathcal
H(z)}{z^p}\notag\\
&=\dfrac{p+\alpha \beta}{\mu \beta}z^{-\frac{p+\alpha\beta}{\mu
\beta}}\int_0^z t^{\frac{\alpha+p\beta}{\mu \beta}-1}(\phi_1 \star
\phi_2)(t)dt\quad (z \in \mathbb U).
\end{align}
Since $\phi_j \in \mathcal P(\gamma_j)$ for $j=1,2$, it follows from
\cite{SS} that
\begin{equation*}
(\phi_1 \star \phi_2) \in \mathcal P(\gamma_3)\quad
\left(\gamma_3=1-2(1-\gamma_1)(1-\gamma_2)\right).
\end{equation*}
and the bound $\gamma_3$ is the best possible. Hence, by using Lemma
\ref{lem2} in \eqref{2.4.19}, we deduce that
\begin{align*}
\text{Re}\{\phi_0(z)\}&=\dfrac{p+\alpha \beta}{\mu \beta}\int_0^1s^{\frac{p+\alpha \beta}{\mu \beta}-1}\,\text{Re}\{(\phi_1 \star\phi_2)(sz)\}ds\\
& \geq \dfrac{p+\alpha \beta}{\mu \beta}\int_0^1 s^{\frac{p+\alpha
\beta}{\mu\beta}-1}\left(2\gamma_3-1+\dfrac{2(1-\gamma_3)}{1+s|z|}\right)ds\\
& > \dfrac{\alpha+p \beta}{\mu \beta}\int_0^1 s^{\frac{p+\alpha \beta}{\mu\beta}-1}\left(2\gamma_3-1+\dfrac{2(1-\gamma_3)}{1+s}\right)ds\\
&=1-\dfrac{4(A_1-B_1)(A_2-B_2)}{(1-B_1)(1-B_2)}\left(1-\dfrac{\alpha+p\beta}{\mu\beta}\int_0^1 \dfrac{s^{\frac{\alpha+p \beta}{\mu\beta}-1}}{1+s}ds\right)\\
&=1-\dfrac{4(A_1-B_1)(A_2-B_2)}{(1-B_1)(1-B_2)}\left\{1-\dfrac{1}{2}\;_2F_1\left(1,1;\dfrac{\alpha+p\beta}{\mu
\beta}+1;\dfrac{1}{2}\right)\right\}\\
& =\eta\quad (z \in \mathbb U).
\end{align*}
When $B_1=B_2=-1$, we consider the functions $f_j \in \mathcal A_p$
satisfying the hypothesis \eqref{2.4.5} and defined by
\begin{equation*}
\Theta_p^m(\alpha,\beta)f_j(z)=\dfrac{p+\alpha \beta}{\mu
\beta}z^{p-\frac{p+\alpha\beta}{\mu \beta}}\int_0^z
t^{\frac{\alpha+p\beta}{\mu
\beta}-1}\left(\dfrac{1+A_jt}{1-t}\right)dt\quad (j=1,2; z
\in\mathbb U).
\end{equation*}
It follows from \eqref{2.4.19} and Lemma \ref{lem2} that
\begin{align*}
\phi_0(z)& =\dfrac{p+\alpha \beta}{\mu
\beta}z^{p-\frac{p+\alpha\beta}{\mu \beta}}\int_0^1
t^{\frac{\alpha+p\beta}{\mu
\beta}-1}\left(1-(1+A_1)(1+A_2)+\dfrac{(1+A_1)(1+A_2)}{1-sz}\right)ds\\
& =
1-(1+A_1)(1+A_2)+(1+A_1)(1+A_2)(1-z)^{-1}\,_2F_1\left(1,1;\dfrac{\alpha+p\beta}{\mu
\beta}+1;\dfrac{z}{z-1}\right)\\ & \longrightarrow
1-(1+A_1)(1+A_2)+(1+A_1)(1+A_2)(1-z)^{-1}\,_2F_1\left(1,1;\dfrac{\alpha+p\beta}{\mu
\beta}+1;\dfrac{1}{2}\right)
\end{align*}
as $z \to -1$, which evidently completes the proof of Theorem
\ref{theorem 2.4.6}.
\end{proof}

\section{Majorization properties}

In this section, we establish  majorization properties of functions
belonging to the class $\mathcal S_p^{m}(\alpha,\beta, A, B)$.

\begin{theorem}\label{theorem 2.5.1}
 Let the function $f \in \mathcal A_p$, and
 suppose that the function $g \in \mathcal S_p^{m}(\alpha,\; \beta,\; A,\;
 B)$. If
 \begin{equation}\label{2.5.2}
 \Theta_p^{m+1}(\alpha, \beta)f(z) \ll \Theta_p^{m+1}(\alpha,
 \beta)g(z)\quad (z \in \mathbb U),
 \end{equation}
 then
 \begin{equation}\label{2.5.3}
 \left\vert \Theta_p^{m+1}(\alpha, \beta)f(z)\right\vert \leq \left\vert \Theta_p^{m+1}(\alpha,
 \beta)g(z)\right\vert\quad( |z\ < \widetilde{r}),
 \end{equation}
 where $\widetilde{r}=\widetilde{r}(p,\alpha, \beta, A,B)$ is the
 smallest positive root of the equation
 \begin{equation}\label{2.5.4}
(\alpha+p\beta)|A|r^3-((\alpha+p\beta)+2\beta
|B|)r^2-((\alpha+p\beta)|A|+2\beta)r+(\alpha+p\beta)=0.
\end{equation}
\end{theorem}

\begin{proof}
Since $g \in \mathcal S_p^{m}(\alpha, \beta, A,
 B)$, it follows that
 \begin{equation*}
\Theta_p^{m+1}(\alpha, \beta)g(z)=\left(\dfrac{1+B
w(z)}{1+Aw(z)}\right)\Theta_p^{m+2}(\alpha, \beta)g(z)\quad (z \in
\mathbb U),
\end{equation*}
where $w$ is analytic in $\mathbb U$ with $w(0)=0$ and $|w(z)| \leq
1$ for all $z \in \mathbb U$. Thus,
\begin{equation}\label{2.5.5}
\left\vert \Theta_p^{m}(\alpha, \beta)g(z)\right\vert \leq
\dfrac{1+|B||z|}{1-|A||z|}\left\vert \Theta_p^{m+1}(\alpha,
\beta)g(z)\right\vert.
\end{equation}
 From \eqref{2.5.3} with the aid of \eqref{2.5.2}, we get
\begin{equation}\label{2.5.6}
\Theta_p^{m+1}(\alpha, \beta)f(z)=\varphi(z)\Theta_p^{m+1}(\alpha,
\beta)g(z)\; (z \in \mathbb U),
\end{equation}
where $\varphi$ is analytic in $\mathbb U$ and $|\varphi(z)| \leq 1$
in $\mathbb U$. Differentiating both the sides of\eqref{2.5.6} and
using the identity \eqref{eq2:eq6} for the functions $f$ and $g$ in
the resulting equation, we deduce that
\begin{equation}\label{2.5.7}
\left\vert \Theta_p^{m+2}(\alpha, \beta)f(z)\right\vert \leq
|\varphi(z)|\left\vert \Theta_p^{m+2}(\alpha,
\beta)g(z)\right\vert+\dfrac{\beta|z|}{\alpha+p\beta}|\varphi^{\prime}(z)|\left\vert
\Theta_p^{m+1}(\alpha, \beta)g(z)\right\vert.
\end{equation}
Now, by using the following estimate \cite{NEHARI}
\begin{equation*}
|\varphi^{\prime}(z)|\leq \dfrac{1-|\varphi(z)|^2}{1-|z|^2}\quad (z
\in \mathbb U)
\end{equation*}
and \eqref{2.5.5} in \eqref{2.5.7}, we get
\begin{equation*}
\left\vert\Theta_p^{m+2}(\alpha, \beta)f(z)\right\vert \leq
\left\{|\varphi(z)|+\dfrac{1-|\varphi(z)|^2}{1-|z|^2}\dfrac{\beta(1+|B||z|)|z|}{(\alpha+p\beta)(1-|A||z|)}\right\}\left\vert
\Theta_p^{m+2}(\alpha, \beta)g(z)\right\vert
\end{equation*}
which upon setting $|z|=r$ and $|\phi(z)|=x\,(0\leq x \leq 1)$
yields
 the inequality
\begin{equation}\label{2.5.8}
\left\vert\Theta_p^{m+1}(\alpha, \beta)f(z)\right\vert \leq
\dfrac{\Psi(x)}{(1-r^2)\{(\alpha+p\beta)(1-|A|r)\}}\left\vert
\Theta_p^{m+1}(\alpha, \beta)g(z)\right\vert,
\end{equation}
where
\begin{equation*}
\Psi(x)=-\beta r(1+|B|r)x^2+(\alpha+p\beta)(1-r^2)(1-|A|r)x+\beta
r(1+|B|r)
\end{equation*}
We note that the function $\Psi$ takes its maximum value at $x=1$
with $\widetilde{r}=\widetilde{r}(p,\; \alpha,\; \beta,\; A,\;B)$
the smallest positive root of the equation \eqref{2.5.4}.
Furthermore, if $0 \leq y \leq \widetilde{r}(p,\; \alpha,\; \beta,\;
A,\;B)$, then the function
\begin{equation*}
\Psi(x)=-\beta y(1+|B|y)x^2+(\alpha+p\beta)(1-y^2)(1-|A|y)x+\beta
y(1+|B|y)
\end{equation*}
increases in the interval $0 \leq x \leq 1$, so that
\begin{equation*}
\Psi(x) \leq \Psi(1)=(\alpha+p\beta)(1-y^2)(1-|A|y).
\end{equation*}
 Thus, in
view of the above fact and \eqref{2.5.8}, we get the assertion
\eqref{2.5.3} of Theorem \ref{theorem 2.5.1}
\end{proof}

Letting $\alpha=p+\ell-p\lambda$ and $\beta=\lambda$ in Theorem
\ref{theorem 2.5.1}, we have

\begin{corollary}\label{corollary 2.5.1}
Let the function $f\in \mathcal A_p$ and the function $g \in
\mathcal S_p^{m}(\alpha,\; \beta,\; A,\;
 B)$ satisfies
\begin{equation*}
\text{Re}\left\{\dfrac{\mathcal J_p^{m+2}(\lambda,
\ell)g(z)}{\mathcal J_p^{m+1}(\lambda, \ell)g(z)}\right\} \prec
\dfrac{1+Az}{1+Bz}\quad (z \in \mathbb U).
\end{equation*}
If $f(z) \ll g(z)$ in $\mathbb U$, then
\begin{equation*}
\left\vert\mathcal J_p^{m+2}(\lambda, \ell)f(z)\right\vert \leq
\left\vert\mathcal J_p^{m+2}(\lambda, \ell)g(z)\right\vert \quad
\text{for}\quad |z| \leq \widetilde{r}(p,\ell,\lambda,A,B),
\end{equation*}
where $\widetilde{r}(p,\;\ell,\;\lambda,\;A,\;B)$ is the smallest
positive root of the equation
\begin{equation*}
(p+\ell)|A|r^3-(p+\ell+\lambda
|B|)r^2-((p+\ell)|A|+2\lambda)r+p+\ell=0.
\end{equation*}
\end{corollary}

\begin{remark}\label{remark 2.5.1}
Setting $A=1-(2\eta/p), B=-1$ and $m=-1\;\alpha=0$ in Theorem
\ref{theorem 2.5.1}, equation \eqref{2.5.4} becomes
\begin{equation*}
|p-2\eta|r^3-(p+2)r^2-(|p-2\eta|+2)r+p=0.
\end{equation*}
It is easily seen that $r=-1$ is a solution of the above equation
and the other two roots can be obtained by solving
\begin{equation}\label{2.5.9}
|p-2\eta|r^2-(|p-2\eta|+p+2)r+p=0.
\end{equation}
So, we can easily find the smallest positive root of \eqref{2.5.9}.
\end{remark}

In view of the above remark, we deduce the following result which,
in turn, yields the corresponding work of MacGregor \cite[p.96,
Theorem 1B]{mag3} for $p=1$ and $\eta=0.$

\begin{corollary}\label{cor 2.5.2}
Let the functions $f \in \mathcal A_p$ and $g \in \mathcal
S_p^{*}(\eta)\,(0 \leq \eta < p)$ be such that
\begin{equation*}
f(z) \ll g(z)\quad (z \in \mathbb U),
\end{equation*}
then
\begin{equation*}
|f^{\prime}(z)| \leq |g^{\prime}(z)|\quad \text{for}\quad |z| \leq
\widetilde{r}(p,\eta),
\end{equation*}
where
\begin{equation*}
\widetilde{r}(p,\eta)=\begin{cases}
\dfrac{p+|p-2\eta|+2-\sqrt{(p+|p-2\eta|+2)^2-4p|p-2\eta|}}{2|p-2\eta|},
& \eta \ne \dfrac{p}{2}\\
\dfrac{p}{p+2}, & \eta=\dfrac{p}{2}.
\end{cases}
\end{equation*}
\end{corollary}

Letting $m=-1,\; A=1-(2\eta/p)$ and $B=-1$ in Corollary \ref{cor
2.5.2}, it follows that for $(p-1)/2 \leq \eta < p$

\begin{equation}\label{2.5.10}
\mathcal C_p(\eta) \subset \mathcal S_p^{*}(\rho),
\end{equation}
where
\begin{equation}\label{2.5.11}
\rho=p\left\{{_2F_1\left(1,2(p-\eta);p+1;\dfrac{1}{2}\right)}\right\}^{-1}.
\end{equation}
Thus, in view of \eqref{2.5.10} and Corollary \ref{cor 2.5.2}, we
obtain the following result which, in turn,  give the corresponding
work by MacGregor \cite [p.96, Theorem 1C]{mag3},
 for $p=1$ and $\eta=0$.
\begin{corollary}\label{corollary 2.5.3}
Let $(p-1)/2 \leq \eta < p.$ If $f \in \mathcal A_p$ and $g \in
\mathcal C_p(\eta)$ satisfy
\begin{equation*}
f(z) \ll g(z)\quad (z \in \mathbb U),
\end{equation*} then
\begin{equation*}
|f^{\prime}(z)| \leq |g^{\prime}(z)|\quad \text{for}\quad |z|<
\widetilde{r}(p,\rho),
\end{equation*}
where
\begin{equation*}
\widetilde{r}(p,\rho)=\begin{cases}
\dfrac{p+|p-2\rho|+2-\sqrt{(p+|p-2\rho|+2)^2-4p|p-2\rho|}}{2|p-2\rho|},
& \rho \ne \dfrac{p}{2}\\
\dfrac{p}{p+2}, & \rho=\dfrac{p}{2}
\end{cases}
\end{equation*}
and $\rho$ is given by \eqref{2.5.11}.
\end{corollary}

\bibliographystyle{plain}

\end{document}